\documentclass{colorart}
\usepackage{ProjLib}
\usepackage[numbers,sort&compress]{natbib}

\usepackage{xspace,enumitem}
\usepackage{graphicx}
\usepackage{subcaption}
\usepackage{nlatexmac}
\usepackage{doi,bm}

\numberwithin{equation}{section}

\date{\today}

\newcommand{\vun}{\mathbb{1}}
\newcommand{\ee}{\epsilon}

\usepackage{scalerel}[2014/03/10]
\usepackage{stackengine}

\newcommand{\rdpp}{\R^d_{++}}
\newcommand{\rdp}{\R^d_{+}}
\newcommand{\rdpz}{\R^d_+\setminus\ens{0}}

\newcommand{\sdp}{S^{d-1}_+}

\UseLanguage{English}
\begin{document}
\UseLanguage{EN}

\title{Differentiability of the Leading Lyapunov Exponent \\
  of a linear differential equation with random coefficients \\
Application to the Calculation of the Selection Gradient in Random Environments}
\author{Philippe Carmona}

\address{Laboratoire de Math\'ematiques Jean Leray UMR 6629\\
Université de Nantes, 2 Rue de la Houssini\`ere\\
BP 92208, F-44322 Nantes Cedex 03, France}

\maketitle 

\begin{abstract}
    The study of evolution in temporally fluctuating environments
    often relies on the analysis of Lyapunov exponents, which quantify
    the exponential growth of populations. However, when model
    parameters depend on a stochastic process, calculating the
    selection gradient, a key tool for predicting the evolution of phenotypic traits,becomes a mathematical challenge. While the periodic case has been resolved, a general approach for random environments remains to be developed.

    This article proposes a rigorous method to:
    \begin{enumerate}
        \item Establish the differentiability of the leading Lyapunov exponent $\Lambda(z)$ with respect to a parameter $z$, providing an explicit integral formula for its derivative $\Lambda'(z)$.
        \item Develop a numerical algorithm to compute $\Lambda'(z)$ by solving an extended differential equation.
        \item Apply these results to the analysis of mutant invasion in a resident population at equilibrium, identifying the selection gradient as the derivative $\Lambda'(z)$.
    \end{enumerate}

    \footnote{\today}
\end{abstract}
\tableofcontents

\section{Notations and Introduction}
Let $d \geq 1$ be an integer. We denote
\begin{align*}
  \rdp &= \ens{x \in \R^d : \forall i, x_i \geq 0} \\
  \rdpp &= \ens{x \in \R^d : \forall i, x_i > 0} \\
  \sdp &= \ens{x \in \rdp: \sum_i x_i = 1}.
\end{align*}
For $x \in \R^d$, we denote $\valabs{x} = \sum_i \valabs{x_i}$, so that if $x \in \rdp$, we have $\valabs{x} = \sum_i x_i = \crochet{x, \vun}$ with $\vun = (1, 1, \ldots, 1)^T$. We denote by $\Mrond$ the set of \emph{Metzler} matrices (also called cooperative matrices): these are the real $d \times d$ matrices whose off-diagonal coefficients are non-negative. A matrix $M \in \Mrond$ is called \emph{irreducible} if for all $1 \leq i, j \leq d$, there exists $n \geq 2$ and $i_1 = i, i_2, \ldots, i_n = j$ such that for all $k \in \ens{1, \ldots, n-1}$, $M_{i_k i_{k+1}} > 0$ (equivalently, if all coefficients of $e^M$ are strictly positive).

Let $(\omega_t)_{t \geq 0}$ be a Feller process with values in the compact metric space $S$, uniquely ergodic, meaning it has a unique invariant probability $\mu$. Given a continuous function $A: S \to \Mrond$, we consider the differential equation
\begin{equation}\label{eq:edolin}
  \frac{dy}{dt} = A(\omega_t) y.
\end{equation}
If, moreover, the average matrix $\bar{A} = \int A(s) \, d\mu(s)$ is irreducible, \citet{Benaim2023} proved the existence of a real number $\Lambda$, called the \emph{leading Lyapunov exponent}, such that for any initial condition $y \in \rdpz$ of equation \eqref{eq:edolin}, we have almost surely
\begin{equation}
  \lim_{t \to +\infty} \frac{1}{t} \log(\valabs{y(t)}) = \Lambda.
\end{equation}

We will consider a function $A(z, s)$ where $A: I \times S \to \Mrond$ is continuous, $C^1$ in the variable $z$ on the open interval $I$, and we denote $\Lambda(z)$ as the leading Lyapunov exponent of the function $s \to A(z, s)$, when it exists. The goal of our study is:
\begin{enumerate}
\item To give  sufficient conditions  for the function $\Lambda(z)$ to be continuously differentiable on $I$, providing an integral formula for the derivative $\Lambda'(z)$.
  \item To establish a limit formula that allows the numerical
  computation of $\Lambda'(z)$ by solving a differential equation that
  is a $2d$-dimensional extension of \eqref{eq:edolin}. To this end, we shall need
 a stronger integrability assumption.
  \item To apply these results  to the study of invasion in a population model in a random environment, identifying the \emph{selection gradient} whose sign indicates the direction of selection.
  \item To show that  this result   generalizes the results established in periodic environments by \citet{LioGan2022}.
\end{enumerate}

The problem of determining the differentiability with respect to
parameters of the Lyapunov exponent is an old problem, treated in a
very general and elegant way for the case of IID products of random
matrices (see
\citet{Hennion91,Lepage89,Ledrappier82,Peres92}) or more
general products of random matrices( see
\citet{Ruelle79,LudGunVol94}). The main interest of our study is to
treat the case of a linear differential equation with random
coefficients, and to provide an explicit numerical algorithm for
computing the derivative, which, to our knowledge, is new.

\bigskip

The first result of this paper, Theorem \ref{thm:main}, derives from
standard results on Lyapunov exponent, and points 1) and 2) are direct
consequences of 
results on Lyapunov exponents due to \citet{MierShen13}, while point
3) is derived from results established by \citet{LudGunVol94} for
products of random matrices.

\section{A first result on differentiability}

Let $(\Omega,\Frond,\PP)$ be a probability space, and  $(\omega_t)_{t
  \geq 0}$ be a Feller process with values in the compact metric space
$S$, uniquely ergodic, meaning it has a unique invariant probability
$\mu$.
Since $(\omega_t)_{t\ge 0}$ is Feller continuous, without loss in generality we can assume that $\Omega$ is the space of
càdlàg functions $\omega: \R_+\to S$ equiped with the Skorokhod
topology and Borel sigma field (see e.g. \cite[Theorem 19.15]{Kal2021}
), on which we consider the time shift
\begin{equation*}
  \theta_t \omega (s) = \omega(t+s) \qquad(t,s\ge 0)\,.
\end{equation*}
Hence $(\Omega,\Frond,\PP)$ is a standard/Lebesgue probability space.

We are given a continuous function $A: I \times S \to \Mrond$ with $I = [a, b]$, $a < b$, and $S$ a compact metric space such that for all $s \in S$, $z \to A(z, s)$ is continuously differentiable with derivative $\partial_z A(z, s)$. We denote by $(A_{ij}(s, z))_{1 \leq i, j \leq d}$ the coefficients and we consider the random variables
\begin{equation}
  \Gamma_z(t,\omega)= \int_0^tm_z(\omega_r)\, dr\,,\quad\text{with } m_z(s) = \inf_{i \neq j} A_{ij}(z, s)\,.
\end{equation}

\begin{hypothesis}\label{hyp:un}
  \begin{enumerate}
    \item For all $z$, the average matrix $\bar{A}(z) = \int A(z, s) \, d\mu(s)$ is irreducible.
    \item The process $(\omega_t)_{t \geq 0}$ is right-continuous with left limits, stationary, and ergodic with invariant law $\mu$.
    \item For all $z$, 
    $\esp{\unsur{\Gamma_z(1,\omega)}} < +\infty$.
  \end{enumerate}
\end{hypothesis}
We can now state our main result. The first part, items 1) and 2), is
already  proved in \citet[Theorem4.1]{MierShen13} (see also
\citet[Theorem 3.1]{LudGunVol94} for a discrete time analog). 
\begin{theorem}\label{thm:main}
  Under Hypothesis  \ref{hyp:un},
  for all $z$, there exists a measurable
set $\tilde{\Omega}_z$ such that $\prob{\tilde{\Omega}_z} = 1$, a real
number $\Lambda(z)$
and two random variables $U_z, V_z$ with values in $\rdpp$, such that for all $\omega \in \tilde{\Omega}_z$:
\begin{enumerate}
  \item $1 = \valabs{U_z(\omega)} = \crochet{V_z(\omega), U_z(\omega)}$.
  \item $\Lambda(z)$ is the leading Lyapunov exponent of the differential equation
  \begin{equation}\label{eq:edoz}
    \frac{dy}{dt} = A(z, \omega_t) y(t),
  \end{equation}
  and we have the integral formula
  \begin{equation}
    \Lambda(z) = \esp{\crochet{A(z, \omega_0) U_z, \vun}}.
  \end{equation}

    \item The function $z \to \Lambda(z)$ is differentiable almost
    everywhere on $(a, b)$ with derivative
  \begin{equation}\label{eq:devlyapuesp}
    \Lambda'(z) = \esp{\crochet{\partial_z A(z, \omega_0) U_z, V_z}}.
  \end{equation}
\end{enumerate}
\end{theorem}
\begin{remark}
  When $A(z,s)=A(z) \in\Mrond$  is
  irreducible, then formula \eqref{eq:devlyapuesp} reduce to
  \begin{equation*}
    \Lambda'(z)= \crochet{A'(z) U_z, V_z}
  \end{equation*}
  and is easy to prove directly (for example  by  a perturbation approach).
\end{remark}

\section{A collection of results for products of random matrices}
\label{sec:tempsdiscret}
We collect results from sections 3 and 4 of \citet{LudGunVol94}, and
apply them to get an integral formula for the derivative of the top
Lyapunov exponent.

Let $\Mrond_+$ be the set of $d\times d$ matrices with positive
coefficients, let $(\Omega,\Frond,\PP,\theta)$ be a dynamical system,
and let $A:\Omega \to \Mrond_+$ measurable such that $\log^+ M_A,
\log^+ (\unsur{m_A}) \in L^1(\PP)$ with
\begin{equation}
  m_A(\omega) = \min_{1\le i,j\le d} A_{ij}(\omega)\,,\quad
  M_A(\omega) = \max_{1\le i,j\le d} A_{ij}(\omega)\,
\end{equation}
We consider the cocycle
\begin{equation}
  \phi_n(\omega) = A(\theta^{n-1}\omega) \cdots A(\omega)\,.
\end{equation}
By Theorem 3.1 of  \citet{LudGunVol94} there exists a unique positive unit
random vector $U: \Omega \to \rdpp$, $\crochet{U,\vun}=1$, and a positive
random scalar $q:\Omega\to(0,\infty)$, with $q\in L^1(\PP)$ such that almost surely
\begin{equation}
  \label{eq:known1}
  A(\omega) = q(\omega) U(\theta \omega)\,,
\end{equation}
and for all $x\in \rdpp$,
\begin{equation}
  \lim_{n\to +\infty} \unsur{n} \log \valabs{\phi_n(\omega)x} =
  \esp{\log q} =: \lambda(A)\,.
\end{equation}

By Proposition 3.8 of \citet{LudGunVol94}, there exists a unique
positive unit random vector $U^*:\Omega\to \rdpp$,
$\crochet{U^*,\vun}=1$, such that almost surely
\begin{equation}
  \label{eq:known2}
  A^*(\omega) U^*(\theta\omega) = q^*(\omega) U^*(\omega)\,,
\end{equation}
with $q^*(\omega) = \valabs{ A^*(\omega) U^*(\theta\omega)} \in
L^1(\PP)$ such that
\begin{equation}
\frac{q^*(\omega)}{q(\omega)}=\frac{\crochet{U(\theta\omega),U^*(\theta\omega)}}{\crochet{U(\omega),U^*(\omega)}}\,.
\end{equation}

\medskip
Let $X=\ens{1,\ldots,d}^\N$, $\sigma$ be the shift on $X$, $\Brond$
the Borel sigma field on $X$. Let $C(X)$ be the space of real valued
continuous functions on $X$ equipped with the sup norm
$\norme{.}_{\infty}$ and $L^1(\Omega,C(x))$ the space of integrable
random continuous functions with the norm
\begin{equation}
  \norme{f}_1 = \int \norme{f(\omega)}_\infty \, d\PP(\omega)\,.
\end{equation}
By Proposition 4.5 and 4.8 of  \citet{LudGunVol94}, we can define the
\emph{pressure} function $\pi_\sigma$ on $L^1(\Omega,C(x))$, which
satisfies
\begin{equation}
  \lambda(A) = \pi_\sigma(\varphi_A)\,,\quad \text{with}\quad
  \varphi_A(\omega)x = \log A_{x_1x_0}(\omega)\,.
\end{equation}
Moreover, the function $\pi_\sigma$ is Gateaux differentiable at
$\varphi_A$,
\begin{equation}
  \frac{d}{dt}_{\mid t=0}(\varphi_A + t g) = \int_{X\times \Omega} g \,
  d\mu_A \quad (g \in L^1(\Omega,C(X))\,,
\end{equation}
with $\mu_A$ the \emph{equilibrium measure} associated to $\varphi_A$,
that is the unique probability measure on $(X\times
\Omega,\Brond\otimes \Frond)$ such that 
\begin{equation}
  \mu(dx,d\omega)= d\mu_\omega(x) d\PP(\omega)\,,
\end{equation}
with
\begin{gather}
  \mu_\omega(x_0=y_0,\ldots,x_n=y_n) = p_{y_0}(\omega) \prod_{j=1}^n
  p_{y_{j-1}y_j}(\theta^{j-1}\omega)\,,\\
  p_{ij}(\omega) = \frac{A_{ji}(\omega) V_j(\theta\omega)}{q^*(\omega)
    V_i(\omega)}\,,\quad p_i(\omega) = \frac{U_i(\omega)
    V_i(\omega)}{\crochet{U(\omega),V(\omega)}}\,.
\end{gather}

\bigskip
Let $A:(a,b)\times \Omega \to \Mrond_+$ be measurable such that for
any $z\in(a,b)$, $\log^+M_{A(z,.)},\log^+(\unsur{m_{A(z,.)}}) \in
L^1(\PP)$. We can consider the top Lyapunov exponent
\begin{equation}
  \Lambda(z) = \lambda(A(z,.)) = \pi_\sigma(\varphi_{A(z,.)})\,.
\end{equation}

Assume that for all $\omega$, the function $z\to A(z,\omega)$ is
differentiable at $z_0 \in(a,b)$, with derivative $\partial_z
A(z_0,\omega)$ such that
\begin{equation}\label{eq:intassumption}
  T:(\omega,x) \to \frac{\partial_z A(z_0,\omega)_{x_1 x_0}}{A_{x_1
      x_0}(\omega)} \in L^1(\Omega,C(X))\,.
\end{equation}
Then, by the chain rule, $\Lambda(z)$ is differentiable at $z_0$ with
derivative
\begin{equation}
  \Lambda'(z_0) = \int T(\omega,x) d\mu_{A(z_0,.)}(\omega,x).
\end{equation}
We let $U(\omega)=U_{z_0}(\omega)$ and $U^*(\omega)=U^*_{z_0}(\omega)$ be
the random unit vectors associated to $A(z_0,\omega)$ and
$A^*(z_0,\omega)$ in equations \eqref{eq:known1}, \eqref{eq:known2} ;  we can
expand the derivative  formula to obtain
\begin{align}
  \Lambda'(z_0)&=\esp{\sum_{i,j} p_i(\omega) p_{ij}(\omega)
    \frac{\partial_z A(z_0,\omega)_{ji}}{A_{ji}(\omega)}} \notag\\
    &=\esp{\unsur{q^*(\omega) \crochet{U(\omega),U^*(\omega)}}
        \sum_{i,j} U_i(\omega) U^*_j(\theta\omega) \partial_z        A(z_0,\omega)_{ji}}\notag \\
      &=\esp{\unsur{q^*(\omega) \crochet{U(\omega),U^*(\omega)}}
          \crochet{\partial_zA(z_0,\omega) U(\omega),
            U^*(\theta\omega)}}\,. \label{eq:fromderivdiscrete}
\end{align}

\section{Proof of Theorem \ref{thm:main}}
\subsection{Exponential Separation}
Throughout this section, the parameter $z \in (a, b)$ is fixed, and we
will often omit the dependence on $z$ to avoid overloading the
notation.

Since $s \to A(s)$ is continuous, according to Theorem 2.2.1 of
\citet{ArnoldRDS}, the pathwise random differential equation
\begin{equation}\label{eq:xedolin}
  \frac{dx(t)}{dt} = A(\omega_t) x(t),
\end{equation}
uniquely generates through its maximal solution a random dynamical 
system $\Phi$ over $\theta$:
$\Phi(t, \omega)$ is the fundamental solution associated with \eqref{eq:xedolin}, i.e., the solution of
\begin{equation}
  \frac{d \Phi(t, \omega)}{dt} = A(\omega_t) \Phi(t, \omega), \quad \Phi(0, \omega) = I_d.
\end{equation}

Since the average matrix $\bar{A}$ is irreducible, Lemma 6 (i) of \cite{Benaim2023} implies that almost surely, there exists an integer $N$ such that for all $t \geq N$, the matrix $\Phi(t, \omega)$ has all its coefficients strictly positive.

We will apply Theorem 4.1 of \citet{MierShen13}. Condition (O1) is
satisfied because the matrices $A(s)$ are cooperative. Condition (O2)
of boundedness is also satisfied : indeed by considering $a<c<d<b$,
thanks to the continuity of $(z,s)\to A(z,s)$ on the compact
$[c,d]\times S$, we can assume without loss in generality that  $C_A= \sup_{i, j, s} \valabs{A(s)_{ij}} < +\infty$, and thus
\begin{equation*}
  \exp\etp{\int_0^1 \sum_{l=1}^d \max_{1\le j\le d} A_{lj}(\omega_t)
    \, dt} \leq e^{C_A dt}.
\end{equation*}
Furthermore the assumption $\esp{\unsur{\Gamma_z}}< +\infty$ implies
condition (O3).

We deduce the existence of  real numbers $\Lambda$, $\sigma \in (0, +\infty)$, random variables $U, V$ with values in $\rdpp$, a measurable set $\tilde{\Omega}$ of probability 1 such that for all $\omega \in \tilde{\Omega}$:
\begin{enumerate}
  \item $\crochet{U(\omega), \vun} = \crochet{V(\omega), U(\omega)} = 1$.
  \item Let $q(t, \omega) = \valabs{\Phi(t, \omega) U(\omega)}$. Then,
  \begin{equation*}
    \Phi(t, \omega) U(\omega) = q(t, \omega) U(\theta_t \omega), \quad \Phi^*(t, \omega) V(\theta_t \omega) = q(t, \omega) V(\omega).
  \end{equation*}
  \item $\Lambda$ is the generalized Lyapunov exponent. For all $x \in \rdpp$,
  \begin{align}
    \Lambda &= \lim_{t \to +\infty} \frac{1}{t} \log \norme{\Phi(t, \omega)} = \lim_{t \to +\infty} \frac{1}{t} \log \valabs{\Phi(t, \omega) x} \notag\\
    &= \lim_{t \to +\infty} \frac{1}{t} \log \norme{\Phi^*(t, \omega)} = \lim_{t \to +\infty} \frac{1}{t} \log \valabs{\Phi^*(t, \omega) x} \notag\\
    &= \lim_{t \to +\infty} \frac{1}{t} \ln q(t, \omega). \label{eq:limqt}
  \end{align}
  \item There is exponential separation, that is
  \begin{equation}\label{eq:expsep}
    \lim_{t \to +\infty} \frac{1}{t} \log \frac{\norme{\Phi(t, \omega)_{\mid V(\omega)^\perp}}}{q(t, \omega)} = \lim_{t \to +\infty} \frac{1}{t} \log \frac{\norme{\Phi^*(t, \omega)_{\mid U(\omega)^\perp}}}{q(t, \omega)} = -\sigma.
  \end{equation}
\end{enumerate}

\begin{lemma}
We have the integral formula
\begin{equation}
  \Lambda = \esp{\crochet{A(\omega_0) U(\omega), \vun}}.
\end{equation}
\end{lemma}

\begin{proof}
We have
\begin{align*}
  q(t, \omega) &= \crochet{\Phi(t, \omega) U(\omega), \vun} \\
  &= \crochet{U(\omega), \vun} + \int_0^t \crochet{A(\omega_s) \Phi(s, \omega) U(\omega), \vun} \, ds \\
  &= 1 + \int_0^t q(s, \omega) \crochet{A(\omega_s) U(\theta_s \omega), \vun} \, ds.
\end{align*}
Consequently, by the ergodic theorem, we have the almost sure limit,
\begin{align*}
  \frac{1}{t} \ln q(t, \omega) = \frac{1}{t} \int_0^t \crochet{A(\omega_s) U(\theta_s \omega), \vun} \, ds \to \esp{\crochet{A(\omega_0) U(\omega), \vun}},
\end{align*}
which gives the desired identity thanks to \eqref{eq:limqt}.
\end{proof}

\subsection{Some useful bounds}
\begin{lemma}\label{lem:bornesphit}
  With $C_A=\sup_{s,i,j}\valabs{A_{i,j}(s)}$ we have
  \begin{align*}
    \Phi_{ij}(t,\omega) &\le e^{C_A d t}\\
    \Phi_{ij}(t,\omega) &\ge e^{-C_A t} \inf(1,\Gamma(t,\omega))\,.
  \end{align*}
  
\end{lemma}
\begin{proof}
  We have $\phi_{ij}(t,\omega)\ge 0$ and $\norme{M}_{\infty} = \sup_{i,j}\valabs{M_{ij}}$
  \begin{equation}
    \frac{d}{dt} \Phi_{ij}(t,\omega) = \sum_{k} A_{ik}(\omega_t)
    \Phi_{kj}(t,\omega) \le C_A d  \norme{\Phi(t,\omega)}_\infty\,.
  \end{equation}
  We can conclude by Gronwall Lemma that $
  \norme{\Phi(t,\omega)}_\infty \le e^{C_A d t}$ since
  $\Phi_{ij}(0,\omega) = 1_{i=j}$.

  Since $A_{ij}(s)\ge 0$ if $i\neq j$,
  \begin{equation}
     \frac{d}{dt} \Phi_{ii}(t,\omega) = \sum_{k} A_{ik}(\omega_t)
    \Phi_{ki}(t,\omega)\ge A_{ii}(\omega_t)
    \Phi_{ii}(t,\omega)\ge -C_A \Phi_{ii}(t,\omega),
  \end{equation}
  we obtain that $\Phi_{ii}(t\omega) \ge e^{- C_A t}$.

  Eventually, if $i\neq j$,
\begin{equation*}
  \frac{d}{dt} \Phi_{ji}(t, \omega) \geq A_{jj}(\theta_t \omega) \Phi_{ji}(t, \omega) + A_{ji}(\theta_t \omega) \Phi_{ii}(t, \omega)
\end{equation*}
implies that
\begin{equation*}
  \Phi_{ji} (t, \omega)\geq \int_0^t e^{\int_s^t
    A_{jj}(\theta_\tau \omega) d\tau} A_{ji}(\theta_s \omega)
  \Phi_{ii}(s, \omega) \, ds \geq e^{-C_A t} \int_0^t
  A_{ji}(\theta_s \omega) \, ds \ge t e^{-C_A t} \Gamma.
\end{equation*}
  
\end{proof}
  \begin{corollary}\label{lem:intphiazun} If $\esp{\unsur{\Gamma_z(1,\omega)}}<
      +\infty$ then
    the random variables $\log^+(M_{\phi_A(z,1,\omega)})$ and
    $\log^+(\unsur{m_{\phi_A(z,1,\omega)}})$ are in $L^1(\PP)$.
  \end{corollary}

  We let
  \begin{equation}
    U^*(\omega) = \frac{V(\omega)}{\crochet{V(\omega),\vun}}\,,\quad
    I(\omega) = \inf_i U_i(\omega)\,, \quad I^*(\omega) = \inf_i U^*_i(\omega).
  \end{equation}
  \begin{corollary}\label{cor:miniomega}
    For a constant $C>0$, we have
    \begin{equation}
      I(\theta \omega)  \ge C \inf(1,\Gamma(1,\omega))\,, \quad
      I^*(\omega) \ge C \inf(1,\Gamma(1,\theta\omega))\,, \quad R(\omega) \ge C \inf(1,\Gamma(1,\theta\omega))\,.
    \end{equation}
  \end{corollary}
  \begin{proof}
    We start from the identity
    \begin{equation}
      \Phi(1,\omega) U(\omega) = q(1,\omega) U(\theta \omega)\,.
    \end{equation}
    Therefore
    \begin{equation}
      q(1,\omega) U_i(\theta\omega) = \sum_j \Phi_{ij}(1,\omega) U_j(\omega)
      \ge e^{- K_A} \inf(1,\Gamma(1,\omega)) \sum_j U_j(\omega) = e^{- K_A} \inf(1,\Gamma(1,\omega)) \,,
    \end{equation}
    and
    \begin{equation}
      q(1,\omega) U_i(\theta\omega) = \sum_j \Phi_{ij}(1,\omega) U_j(\omega)
      \le e^{d C_A} \sum_j U_j(\omega) = e^{d C_A}\,.
    \end{equation}
    Summing on index $i$ the last inequality we get
    \begin{equation}
      q(1,\omega) \le d e^{d C_A}\,,
    \end{equation}
    and therefore
    \begin{equation}
      I(\theta\omega) = \inf_i U_i(\theta\omega) \ge C \inf(1,\Gamma(1,\omega))\,.
    \end{equation}
    Similarly, starting from the equality
    \begin{equation}
      \Phi^*(1,\omega) U^*(\theta \omega) = q^*(1,\omega) U^*(\omega),
    \end{equation}
    we obtain
    \begin{equation}
      I^*(\omega) \ge C \inf(1,\Gamma(1,\theta\omega))\,.
    \end{equation}
    Eventually, since $U(\omega)$ and $U^*(\omega)$ are positive
    unitary random vectors, $\norme{U(\omega)}=\norme{U^*(\omega) }=1$,
        so
        \begin{equation}
          R(\omega) = \sum_i U_i(\omega) U^*_i(\omega) \ge I(\omega)
          \sum_i U^*_i(\omega) = I(\omega)\,,
        \end{equation}
        and likewise $R(\omega) \ge I^*(\omega)$.
  \end{proof}

\subsection{Differentiability of $\Lambda$ and the integral formula
  \eqref{eq:devlyapuesp}}
Thanks to Lemma \ref{lem:intphiazun} we can apply the result of section \ref{sec:tempsdiscret} to the
matrix random function $(z,\omega) \to \phi_A(z,1,\omega)\in
\Mrond_+$. The discrete random cocycle associated is
\begin{equation}
  \Phi_A(z,1,\theta^{n-1}\omega) \ldots \Phi_A(z,1,\omega) = \Phi_A(z,n\omega)\,
\end{equation}
so the top Lyapunov of this cocycle is exactly $\Lambda(z)$:
\begin{equation}
  \lambda(\Phi_A(z,1,\omega)) = \lim_{n\to +\infty} \unsur{n} \log
  \norme{\Phi_A(z,n,\omega)} = \Lambda(z)\,.
  \end{equation}

The unit vectors associated to $\Phi_A(z,1,\omega)$ and
$\Phi_A^*(z,1,\omega)$ are $U_z(\omega)$ and $U^*_z(\omega)$. From the
normalisation $\crochet{U_z,V_z}=1$ we deduce that $V_z(\omega) =
\frac{U_z^*(\omega)}{\crochet{U_z(\omega),U^*_z(\omega)}}$. Therefore,
formula \eqref{eq:fromderivdiscrete} is
\begin{equation}\label{eq:exprderivintermediaire}
  \Lambda'(z) = \esp{\frac{\crochet{\partial_z \Phi_A(z,1,\omega) U_z(\omega),V_z(\theta\omega)}}{q(1,z,\omega)}}\,.
  \end{equation}
Of course, we need to check the integrability of the random variable
$T$ defined in \eqref{eq:intassumption}:
\begin{equation}
  \esp{\norme{T(\omega)}_\infty} = \esp{\sup_{i,j}
    \frac{\valabs{\partial_z\Phi_A(z,1,\omega)_{ij}}}{\Phi_A(z,1,\omega)_{ij}}}
  \le C \esp{\unsur{\inf(1,\Gamma_z)}} < +\infty\,,
\end{equation}
where we used Lemma \ref{lem:bornesphit}.

Observe that if $\phi_M$ is the fundamental solution associated with the continuous matrix function $t \to M(t)$:
\begin{equation*}
  \frac{d \phi_M(t)}{dt} = M(t) \phi_M(t), \quad \phi_M(0) = I_d,
\end{equation*}
then we have the integration by parts formula
\begin{equation*}
  \phi_{M+N}(t) = \phi_M(t) + \int_0^t \phi_M(t-s) N(s) \phi_{M+N}(s) \, ds.
\end{equation*}

We deduce that the derivative of the function $z \to \Phi(z, t, \omega)$ is
\begin{equation*}
  \partial_z \Phi(z, t, \omega) = \int_0^t \Phi(z, t-s, \theta_s \omega) \partial_z A(z, \omega_s) \Phi(z, s, \omega) \, ds.
\end{equation*}

To conclude our proof, it remains to show that  expression \eqref{eq:exprderivintermediaire} coincides with $\esp{\crochet{V_z(\omega), \partial_zA(z, \omega_0) U_z(\omega)}}$. 

Observe that

\begin{align*}
  \crochet{V_z(\theta \omega), \partial_z \Phi(1, z, \omega) U_z(\omega)} &= \int_0^1 \crochet{V_z(\theta \omega), \Phi(z, 1-s, \theta_s \omega) \partial_z A(z, \omega_s) \Phi(z, s, \omega) U_z(\omega)} \, ds \\
  &= \int_0^1 q(s, \omega) \crochet{V_z(\theta \omega), \Phi(z, 1-s, \theta_s \omega) \partial_z A(z, \omega_s) U_z(\theta_s \omega)} \, ds \\
  &= \int_0^1 q(s, \omega) \crochet{\Phi^*(z, 1-s, \theta_s \omega) V_z(\theta \omega), \partial_z A(z, \omega_s) U_z(\theta_s \omega)} \, ds \\
  &= \int_0^1 q(s, \omega) q(1-s, \theta_s \omega) \crochet{V_z(\theta_s \omega), \partial_z A(z, \omega_s) U_z(\theta_s \omega)} \, ds \\
  &= q(1, \omega) \int_0^1 \crochet{V_z(\theta_s \omega), \partial_z A(z, \omega_s) U_z(\theta_s \omega)} \, ds,
\end{align*}
since by the cocycle property, $q(1, \omega) = q(s, \omega) q(1-s, \theta_s \omega)$. Consequently, by stationarity,
\begin{align*}
  \Lambda'(z) &= \esp{\int_0^1 \crochet{V_z(\theta_s \omega), \partial_z A(z, \omega_s) U_z(\theta_s \omega)} \, ds} \\
  &= \int_0^1 \esp{\crochet{V_z(\theta_s \omega), \partial_z A(z, \omega_s) U_z(\theta_s \omega)}} \, ds \\
  &= \esp{\crochet{V_z(\omega), \partial A(z, \omega_0) U_z(\omega)}}.
\end{align*}

\xcom{
\begin{lemma}
  \label{lem:intvomega}
  If, for some $\delta>0$, $\esp{\Gamma_z^{-\delta}}< +\infty$, then
  $\esp{\crochet{V(\omega),\vun}^\delta} < +\infty$.
\end{lemma}
\begin{proof}
  For simplicity we omit the dependency on $z$. First observe that
  \begin{equation}
    1 = \crochet{V(\omega),U(\omega)} \ge \crochet{V(\omega),\vun}\, \inf_i U_i(\omega) \,.
    \end{equation}
    On the one hand, since from the proof of Lemma
    \ref{lem:intphiazun} $\Phi_{ij}(1,\omega) \le^{K d}$ we have
    \begin{equation}
      q(1,\omega) U_i(\theta\omega) = \sum_j \Phi_{ij} (1,\omega) U_j(\theta\omega)
        \le e^{K d} \sum_j U_j(\theta\omega) = e^{Kd}\,.
      \end{equation}
      Summing on index $i$ yields
      \begin{equation}
        q(1,\omega) \le d e^{Kd}.
      \end{equation}
      On the other hand, since, from the proof of Lemma
    \ref{lem:intphiazun}, $\Phi_{ij}(1,\omega) \ge e^{-K}
    \inf(1,\Gamma_z)$, we get
    \begin{equation}
      q(1,\omega) U_i(\theta\omega) = \sum_j \Phi_{ij}(1,\omega) U_j(\theta\omega)
        \ge  e^{-K} \inf(1,\Gamma_z)\sum_j U_j(\theta\omega) = e^{-K} \inf(1,\Gamma_z)\,.
      \end{equation}
      Combining all these inequalities, we obtain
      \begin{equation}
        \crochet{V(\theta \omega),\vun} \le C \inf(1,\Gamma_z)^{-1}\,,
      \end{equation}
      and we conclude since $\crochet{V(\theta \omega),\vun}$ is
      distributed as $\crochet{V( \omega),\vun}$.
\end{proof}
}
\subsection{A bound on the exponential separation}
The goal of this section is to give explicit bounds for  the
convergence of $\unsur{q(t,\omega)}\Phi(t,\omega)\vun$ to
$\crochet{V(\omega),x}U(\omega)$ : the exponential factor can be taken
independent of $\omega$, so this is an improvement on the exponential
separation \eqref{eq:expsep}.

To this end we are going to rely heavily on Ruelle's paper \cite{Ruelle79}, especially
the proof of Proposition 3.2. The key ingredient is \emph{compactness}.

Since $A:S\to \Mrond$ is continuous and $S$ compact, it takes its
valeus in a compact set of $d\times d$ matrices, therefore
$\Phi(1,\omega)$ takes its values in a compact $\Krond$ set $\Mrond_+$
of
matrices with strictly positive coefficients. We consider the compact
set $X=\Krond^\N$ endowed with the product topology and its Borel
sigma field. The probability $\rho$ is defined as the image of $\PP$
by the measurable application
\begin{equation}
  G(\omega) = (\Phi(1,\theta^n \omega),n\ge 0)\,.
\end{equation}
The shift map on $X$, defined as $\tau(X_n,n\ge 0)= (x_{n+1},n\ge 0)$
preserves the measure $\rho$, which is ergodic.

We try to be as close as possible to Ruelle's notation, so we consider
the measurable application $T:X\to \Mrond_+$ given by $T(x)=x_0$. The
cone $C(x)$ of Theorem 3.1 of Ruelle is constant and is the positive
orthant.
Thanks to proposition 3.2 of Ruelle, there exists a continuous
function $b:X\to \R^d$, such that $b(x)$ is a unit vector with
positive coordinates that satisfies
\begin{equation}
  T(x) b(x) = \norme{T(x) b(x)} b(\tau x)\,,
\end{equation}
with $\norme{z}=\sum_{i=1}^d \valabs{z_i}$. Hence, the unit positive
random vector
\begin{equation}
  u(\omega) := b(G(\omega))
\end{equation}
satisfies almost surely
\begin{equation}
  \phi(1,\omega) u(\omega) = \norme{\phi(1,\omega) u(\omega)}u(\theta\omega)\,.
\end{equation}
By uniqueness, \citet[Theorem 3.1]{LudGunVol94}, we have $U(\omega) =
u(\omega)$, and therefore $\norme{\phi(1,\omega)
  u(\omega)}u(\theta\omega)=q(1,\omega)$.

Similarly, there exists a continuous function $u^*:X\to \R^d$ such
that $b^*(x)$ is a positive unit vector that satisfies
\begin{equation}
  T^*(x) b^*( \tau x) = \norme{T^*(x) b^*(\tau x)} b^*( x)\,.
\end{equation}
Let $U^*(\omega)= b^*(G(\omega))$. By uniqueness, we have
$$V(\omega) = \frac{U^*(\omega)}{R(\omega)}\,,\quad U^*(\omega)
= \frac{V(\omega)}{S(\omega)}\quad\text{with } R(\omega) = \crochet{U(\omega),U^*(\omega)}\,,
S(\omega)=\crochet{V(\omega),\vun} = \norme{V(\omega)}\,.
$$
We let
\begin{equation}
  I(\omega) = \inf_i U_i(\omega)\,,\quad I^*(\omega) = \inf_i U^*_i(\omega)\,.
\end{equation}

Define
\begin{equation}
  K(x) :=\ens{\xi \in \R^d: \crochet{b^*(x),\xi} = 0, b(x) + \xi \in \rdp}\,,\qquad
  T'(x) \xi= \frac{T(x) \xi}{\norme{T(x) b(x)}}\,.
\end{equation}
Then\footnote{Here is one place where Ruelle uses compactness of $X=\Krond^\N$} there exists $\alpha \in (0,1)$ such that
\begin{equation}
  T'(x) K(x) \subset \alpha K(\tau x)\,.
\end{equation}
If we translate this on $(\Omega,\Frond,\PP)$ we obtain that almost
surely
\begin{equation}\label{eq:ruelleinclusion}
  \unsur{q(1,\omega)} \Phi(1,\omega) K(\omega) \subset \alpha K(\theta \omega)\,,
\end{equation}
with 
\begin{equation}
  K(\omega) = \ens{\xi \in \R^d: \crochet{U^*(\omega),\xi}=0, U(\omega)
    + \xi \in \rdp}\,.
\end{equation}
Since
\begin{equation}
  \Phi(n,\omega) = \Phi(1,\theta^{n-1}\omega)\cdots
  \Phi(1,\omega)\,,\quad q(n,\omega) = q(1,\theta^{n-1}\omega)\cdots q(1,\omega)\,,
\end{equation}
we can iterate the inclusion \eqref{eq:ruelleinclusion}, to obtain
\begin{equation}
  \label{eq:run}
   \unsur{q(n,\omega)} \Phi(n,\omega) K(\omega) \subset \alpha^n K(\theta^n \omega)\,.
 \end{equation}

The following Lemma gives us precise bounds on the exponential separation.
 \begin{lemma}
   We have the decomposition
    \begin{equation}\label{eq:decompoungamma}
    \frac{\Phi(t,\omega) \vun}{q(t,\omega)} =  \unsur{R(\omega)}
    (U(\theta^t \omega) + \gamma_{t,\omega})
  \end{equation}
  with
  \begin{equation}\label{eq:majnormegamma}
  \norme{\gamma_{t,\omega}} \le C \alpha^{\floor{t}}\etp{1 + \unsur{I^*(\theta^{\floor{t}}\omega)}}\,.
\end{equation}

 \end{lemma}
\begin{proof}
From the decomposition
\begin{equation}
  \vun = \unsur{R(\omega)}(U(\omega) +\xi_z(\omega)), \quad
  (\xi_z(\omega) \in K(\omega))\,,
\end{equation}
we obtain
\begin{equation}
  \frac{\Phi(n,\omega)z}{q(n,\omega)} = \unsur{R(\omega)}
  (U(\theta^n\omega) + \alpha^n \xi_{n,\omega})\,,\quad (\xi_{n,\omega} \in K(\theta^n\omega))\,.
\end{equation}

 Since the image $A(S)$ is compact, we have
 \begin{equation}
   \sup_{s\le 1,\omega'\in\Omega} \norme{\Phi(s,\omega')} <
   +\infty\,,\quad \sup_{s\le 1,\omega'\in\Omega} \unsur{q(s,\omega')}
   < +\infty\,.
 \end{equation}
 Therefore, since with $n=\floor{t}$
 \begin{equation}
   \Phi(t,\omega) = \Phi(t-n,\theta^n\omega) \Phi(n,\omega)
   \quad q(t,\omega)= q(t-n,\theta^n \omega) q(n,\omega),
 \end{equation}

 we obtain the decomposition \eqref{eq:decompoungamma},  with
  \begin{equation}
    \gamma_{t,\omega} = \frac{\alpha^{\floor{t}}}{q(t-n,\theta^n
      \omega)}\Phi(t-n,\theta^n\omega) \xi_{n,\omega}\,.
  \end{equation}
  Let $\xi \in K(\omega)$, so that $y=\xi + U(\omega) \in \rdp$ and
  \begin{equation}
    \crochet{y,U^*(\omega)} = \crochet{U(\omega),U^*(\omega)} =
  R(\omega) \le \crochet{\vun,U(\omega)}=1\,.
\end{equation}
Therefore,
\begin{equation}
  I^*(\omega) \norme{y}=I^*(\omega) \crochet{y,\vun} \le
  \crochet{y,U^*(\omega)}\le 1\,,
\end{equation}
and thus
\begin{equation}
  \norme{\xi} \le \norme{U(\omega)} +\norme{y} \le 1 +
\unsur{I^*(\omega)}\quad (\xi \in K(\omega))\,.
\end{equation}
We thus obtain \eqref{eq:majnormegamma}.
\end{proof}
The following Lemma is the dual of the preceding Lemma : there is an
extra $R(\omega)$ factor due to the fact that
$\frac{q^*(t,\omega)}{q(t,\omega)}=\frac{R(\theta^t \omega)}{R(\omega)}$.
\begin{lemma}
  We have the decomposition
  \begin{equation}
    \label{eq:decompoungammastar}
    \frac{\Phi^*(t,\omega) \vun}{q(t,\omega)} = V(\omega) + \gamma^*_{t,\omega}\,,
  \end{equation}
  with
  \begin{equation}
    \label{eq:majnormegammastar}
    \norme{\gamma^*_t(\omega)} \le \frac{C \alpha^{\floor{t}}}{R(\omega)}  \etp{1 + \unsur{I(\theta^{t}\omega)}}\,.
  \end{equation}
\end{lemma}
\begin{proof}
Our starting point is the inclusion
  \begin{equation}
     \frac{\Phi^*(1,\omega) }{q^*(1,\omega)} K^*(\theta \omega) \subset
     \alpha K^*(\omega)\,,
   \end{equation}
   with
   \begin{equation}
     K^*(\omega) = \ens{\xi \in \R^d: \crochet{\xi,U(\omega)}=0,
       U^*(\omega) + \xi \in \rdp}\,,
   \end{equation}
   which we generalize to
   \begin{equation}
      \frac{\Phi^*(n,\omega) }{q^*(n,\omega)} K^*(\theta^n \omega) \subset
     \alpha^n K^*(\omega)\quad(n\in\N)\,.
   \end{equation}
   From the decomposition
   \begin{equation}
     \vun = \unsur{R(\theta^t \omega)}(U^*(\theta^t \omega) +\xi(\omega))
   \end{equation}
   with $\xi(\omega) \in K^*(\theta^t\omega)$,
   we obtain that, since $ \frac{q^*_t(\omega)}{R(\theta^t \omega)} =
   \frac{q(t,\omega)}{R(\omega)}$, with $n=\floor{t}$,
   \begin{align}
     \frac{\Phi^*(t,\omega) \vun}{q(t,\omega)} &= V(\omega) +
     \unsur{R(\omega) q^*(t,\omega)} \Phi^*(t,\omega) \xi(\omega) \\
     &= V(\omega) + \unsur{R(\omega)} \ \frac{\Phi^*(n,\omega) }{q^*(n,\omega)} \circ\frac{\Phi^*(t-n,\theta^n\omega) \xi(\omega)}{q^*(t-n,\theta^n\omega)} \,.
   \end{align}
Since $\xi(\omega) \in K^*(\theta^t \omega)$, we have
\begin{equation}
  \eta(\omega) = \frac{\Phi^*(t-n,\theta^n\omega)
    \xi(\omega)}{q^*(t-n,\theta^n\omega)}  \in K^*(\theta^n \omega)\,.
\end{equation}
Therefore
\begin{equation}
  \gamma^*_t(\omega) = \unsur{R(\omega)} \frac{\Phi^*(n,\omega)
    \eta(\omega)}{q^*(n,\omega)} = \unsur{R(\omega)} \alpha^n \zeta(\omega)
\end{equation}
with $\zeta(\omega) \in K^*(\omega)$,
  from which we obtain the desired bound on $\norme{\gamma^*_{t,\omega}}$.
\end{proof}

\section{Approximation of the derivative of the top Lyapunov exponent}
In addition to Hypothesis~\ref{hyp:un}  that ensures the differentiability of
the top Lyapunov exponent $\Lambda(z)$, we give suffient conditions
for a limit theorem that gives an approximation of the derivative
$\Lambda'(z)$.

\begin{theorem}
  \label{thm:an}
  Assume that Hypothesis~\ref{hyp:un} is satisfied and that for a
  fixed $z\in (a,b)$  there exists $\delta>1$, such that
$\esp{\Gamma_z^{-2\delta}}< +\infty$. Then for any initial value
$(y_0,0)$ with $y_0\in\rdpp$, the solution $(y(t),\partial_zy(t))$ of the random linear
differential system
\begin{align}
  \frac{dy}{dt} &= A(z,\omega_t) y\\
  \frac{d\partial_z y}{dt} &= A(z,\omega_t) \partial_z y + \partial_z A(z,\omega_t)
  y
\end{align}
satisfies, 
\begin{equation}
  \lim_{t\to +\infty} \frac{\crochet{\partial_zy(t),\vun}}{t
    \crochet{y(t  ),\vun}} = \Lambda'(z)\quad a.s.
\end{equation}
\end{theorem}
\begin{remark}
  One can see this theorem to be just a way to interchange derivative
  and limit in the formula
  \begin{equation}
    \frac{d \Lambda(z)}{dz} = \frac{d}{dz}\etp{ \lim_{t\to \infty}
    \unsur{t} \log \crochet{y(t),1}}\,.
  \end{equation}
\end{remark}
\begin{proof}
Without loss in generality we shall assume that $y_0=\vun$. Thanks to
Lemma \ref{eq:decompoungamma}, we have
\begin{equation}
  \lim_{t\to +\infty} \frac{\crochet{y(t),\vun}}{q(t,\omega)}
  =\unsur{R(\omega)}\lim_{t\to +\infty} (1 + \crochet{\vun,
    \gamma_{t,\omega}}) = \unsur{R(\omega)}.
\end{equation}
Our starting point for the limit is the integral formula, a direct
consequence of the linear differential system,
\begin{equation}
  \partial_zy(t) = \intot \phi(t-s,\theta_s \omega) \partial_z
  A(z,\omega_s) y(s), ds\,.
\end{equation}
Since $y(s) = \phi(s,\omega)\vun$ and $q(t,\omega) =
q(t-s,\theta_s\omega) q(s,\omega)$, we obtain, by Lemmas
\ref{eq:decompoungamma} and \ref{eq:decompoungammastar}
\begin{align}
  T(t,\omega) = \unsur{t R(\omega)} \frac{\crochet{\partial_z y(t),\vun}}{q(t,\omega)} &=
      \unsur{t R(\omega)} \intot \crochet{\partial_z A(z,\omega_s)
        \frac{\phi(s,\omega)\vun}{q(s,\omega)},
        \frac{\phi^*(t-s,\theta_s\omega)\vun}{q(t-s,\theta_s\omega)}}\,
        ds \\
        &= \unsur{t} \intot \crochet{\partial_z
          A(z,\omega_s)( U(\theta_s\omega) +
          \gamma_{s,\omega}),V(\theta_s\omega) +
          \gamma^*_{t-s,\theta^s\omega}}\, ds \\
        &= T_{11}(t,\omega) + T_{12}(t,\omega) + T_{21}(t,\omega) + T_{22}(t,\omega)\,,
      \end{align}
      
      with
      \begin{align}
        \norme{\gamma_{s,\omega}}\le C \alpha^{\floor{s}}
          \unsur{I^*(\theta^{\floor{s}})}\,,\quad
          \norme{\gamma^*_{t-s,\theta^s\omega}} \le C
          \frac{\alpha^{\floor{t-s}}}{R(\theta^{s}\omega)} (1+ \unsur{I(\theta^{t}\omega)})
          \end{align}
          and
      \begin{align}
T_{11}(t,\omega) &=     \unsur{t} \intot \crochet{\partial_z
          A(z,\omega_s) U(\theta_s\omega) ,V(\theta_s\omega) }\, ds
        \\
  T_{12}(t,\omega) &=     \unsur{t} \intot \crochet{\partial_z
          A(z,\omega_s) U(\theta_s\omega) ,\gamma^*_{t-s,\theta^s\omega}}\, ds
        \\
        T_{21}(t,\omega) &=  \unsur{t} \intot \crochet{\partial_z
          A(z,\omega_s) \gamma_{s,\omega} ,V(\theta_s\omega) }\, ds  \\
        T_{22}(t,\omega) &=     \unsur{t} \intot \crochet{\partial_z
          A(z,\omega_s) \gamma_{s,\omega} ,\gamma^*_{t-s,\theta^s \omega}}\, ds\,.
      \end{align}
      We shall show that almost surely $T_{11}(t,\omega) \to       \Lambda'(z)$ and that for $(i,j)\neq(1,1)$, $T_{ij}(\omega,t)
      \to 0$.

First, we apply the ergodic theorem to $T_{11}(t,\omega)$. To this end
we need to show that
\begin{equation}
  X_1(\omega) = \crochet{\partial_z A(z,\omega_0,U( \omega),
    V( \omega)} \in L^1(\PP)\,.
\end{equation}
This is indeed the case since $\sup_{z,s} \norme{\partial_z
  A(z,s)}_\infty < +\infty$ , $\norme{U(\omega)}=1$ and thus thanks to
Corollary \ref{cor:miniomega}
\begin{equation}
  \valabs{X_1(\omega)} \le C \norme{V(\omega)} =C \unsur{R(\omega)}
  \le C' \unsur{\inf(1,\Gamma(1,\omega))} \in L^1(\PP)\,.
\end{equation}
Hence, the ergodic theorem yields that  almost surely
$T_{11}(t,\omega) \to \esp{X_1}= \Lambda'(z)$.

Now, similarly,
\begin{align*}
  \valabs{T_{21}(t,\omega)}&\le \frac{C}{t} \intot
  \frac{\alpha^{\floor{s}}}{I^*(\theta^{\floor{s}}\omega)}
  \norme{V(\theta^s \omega)}\, ds \\
  &\le \frac{C}{t} \sum_{k=1}^{\floor{t}}
  \frac{\alpha^k}{I^*(\theta^k\omega)} \int_k^{k+1} R(\theta^s
    \omega)\, ds\\
  & \to 0
\end{align*}
thanks to Lemma \ref{lem:blob} and the ergodic theorem applied to the
random variable $X_{21}^\delta$ with 
\begin{equation}
  X_{21}(\omega) = \unsur{I^*(\omega)}\int_0^1
  R(\theta^s\omega)\, ds\,.
\end{equation}
Indeed, by Corollary \ref{cor:miniomega}, Jensen's inequality and stationarity
\begin{align*}
  \esp{X_{21}^\delta}^2 &\le \esp{\unsur{I^*(\omega)}^{2\delta}}
  \esp{\etp{\int_0^1 R(\theta^s \omega)\,ds}^{2\delta}} \\
  &\le \esp{\unsur{I^*(\omega)}^{2\delta}}\esp{\int_0^1 R(\theta^s
    \omega)^{2\delta}\, ds} \\
  &\le \esp{\unsur{I^*(\omega)}^{2\delta}} \esp{R(\omega)^{2\delta}}
  \\
  &\le C \esp{\inf(1,\Gamma(1,\omega))^{2\delta}}^2 < +\infty\,.
\end{align*}

Similarly, we have
\begin{align*}
  \valabs{T_{12}(t,\omega)} &\le \frac{C}{t} \intot
  \frac{\alpha^{\floor{t-s}}}{R(\theta^{s}\omega)} (1+
  \unsur{I(\theta^{t}\omega)}) ds \\
  &\le  \unsur{I(\theta^{t}\omega)})  \frac{C}{t} \int_0^t
  \alpha^{t-s} \unsur{R(\theta^s)}\, ds\\
  & \to 0 &a.s.\,,
\end{align*}
      by applying Lemma \ref{lem:blob}, and the ergodic theorem to
      $R(\omega)^{-\delta}$ because
      \begin{equation}
        R(\omega)^{-\delta} \le C \inf(1,\Gamma(1,\omega))^{-\delta}
        \in L^1(\PP)\,.
      \end{equation}

Eventually, we have
\begin{align}
  \valabs{T_{22}(t\,\omega) } &\le (1+ \unsur{I(\theta^t\omega)}) \frac{C}{t} \intot
  \alpha^{\floor{s} + \floor{t-s}}  \unsur{I^*(\theta^s \omega)}\,ds
  \\
  &\le \etp{1+ \unsur{I(\theta^t\omega)}} \frac{C\alpha^t}{t} \intot
  \unsur{I^*(\theta^s \omega)}\,ds\\
   & \to 0 &a.s.\,, 
 \end{align}
 by the ergodic theorem since
 \begin{equation}
    \unsur{I^*(\omega)} \le C \unsur{\Gamma(1,\omega)} \in L^1(\PP)\,.
 \end{equation}
\end{proof}

\begin{lemma}
  \label{lem:blob}
  Let $f:\R_+\to \R_+$ measurable such that for some $\ee>0$, the
  following limit exist and is finite
  \begin{equation}
    \lim_{t\to +\infty} \unsur{t} \intot f(s)^{1+\ee}\, ds =: m \in
    \R_+\,.
  \end{equation}

  Then, for any $\alpha>0$,
  \begin{equation}
    \lim_{t\to +\infty} \unsur{t} \intot e^{-\alpha s} f(s)\, ds = 0\,.
  \end{equation}
\end{lemma}
\begin{proof}
  Let $A>0$. Then, as $t\to +\infty$,
  \begin{equation}
    \unsur{t} \intot e^{-\alpha s} f(s) \un{f(s) \le A} \, ds\le
    \frac{A}{t} \intof e^{-\alpha s} \, ds =
    \frac{A}{t\alpha}\to 0.
  \end{equation}
  
  Observe that
  
  \begin{align*}
    \unsur{t} \intot e^{-\alpha s} f(s) \un{f(s) \ge A} \, ds   &\le
    \unsur{t} \intot  f(s) \un{f(s) \ge A} \, ds \\
  &\le \unsur{t} \intot  f(s) \etp{\frac{f(s)}{A}}^\ee \, ds\\
    &\le A^{-\ee}  \unsur{t} \intot f(s)^{1+\ee}\, ds \to A^{-\ee} m\,.
  \end{align*}
  
  Therefore, for any $A>0$,
  \begin{equation}
    \limsup_{t\to +\infty} \unsur{t} \intot e^{-\alpha s} f(s)\, ds
    \le  A^{-\ee} m\,.
  \end{equation}
  We conclude by letting $A\to +\infty$.
\end{proof}

\section{Application to the Calculation of the Selection Gradient}
The purpose of this section is to show how the calculation of the selection gradient, which is the derivative of the leading Lyapunov exponent, allows the analysis of the invasion of a mutant in a resident population at equilibrium (under its stationary measure).

For a more rigorous and precise introduction to the biological
problem, we can only recommend reading introductory articles on
adaptive dynamics
(\url{https://en.wikipedia.org/wiki/Adaptive_dynamics},
\citet{MetGer96}) and invasion analysis (\citet{FerGat95}).

Let us assume that a resident population with trait $z_w$ evolves in a
random environment $(\xi_t)_{t\ge 0}$ and is described by the process
$(X_w(t))_{t\ge 0}$. The central question of invasion biology is wether
the introduction of a \emph{small number of individuals of a mutant species},
with trait $z_m\neq z_w$, can lead to its becoming established.
Thanks to the large population approximations for density dependent
Markov processes of \citet{kurtz78}, we can model the evolution of the
resident population as a system of ordinary differential equations
with random coefficients.
However, the large populations approximations are no longer
appropriate to model  the mutant population (introduced in small numbers).
A natural process approximation, considered e.g. by
\citet{Whittle13,Kendall8} is the branching process in a random
environment.
For example \cite{BallDonelly2} proved for a quite general population process
that if $N$ denotes the typical number of resident susceptibles, the
mutant epidemic process is well approximated up to time $o(\log N)$ by
a branching process.

\bigskip
We shall therefore assume that the evolution of the mutant process is
modeled by a multi-type branching process $(X_m(t))_{t\ge 0}$  in a
random environment $(\omega_t)_{t\ge 0}$.
Observe that the environment $\omega_t=(\xi(t),X_w(t))_{t\ge 0}$ of the mutant branching
process is made by the combination of the original random environment
with the process of the resident population.

Thanks to the now classic results on branching processes (see the
paper by \citet{Lam25}, Remark 2.1 and references therein), the
evolution of the branching process is dictated by  the dynamics of the
vector of means $(Y(t) = \mathbb{E}'\etc{X_m(t, \omega')})_{t \geq 0}$
which is the solution of the differential equation with random coefficients
\begin{equation}
  \frac{dY}{dt} = A(\omega_t) Y(t).
\end{equation}
(In \citet{athmul68}, the matrix $A(\omega_t)$ is called the
infinitesimal generator of the semigroup of mean matrices.)

We assume that we have \emph{weak selection}, that is that the mutant
trait $z_w$ is a small variation of the resident trait, and that we
satisfy Hypothesis \ref{hyp:un}.
The leading Lyapunov exponent $\Lambda(z)$ exists and is a
continuously differentiable function on the interval $I$.

\begin{definition}
The \emph{selection gradient} is defined by $\Lambda'(z_w)$: it determines the direction of selection when the resident's trait is $z_w$.
\end{definition}
Indeed
let us first observe that $\Lambda(z_w) = 0$. In biological terms,
this means that a resident at equilibrium cannot
self-invade. Mathematically, since the process $(\omega_t)_{t\ge 0} = (X_w(t), \xi(t)))_{t \geq 0}$ is stationary, there cannot be exponential growth or decay due to a small variation in the resident population.

Assume  for example, that  $\Lambda'(z_w)>0$ for $z_w$ in the interior of
 $I$.  Then for $\ee$ small enough,
 if  $z_m \in (z_w,z_w+\ee)$, we have $\Lambda(z_m)>0$, and thus the
 mutant branching process $X_m$ is supercritical and  has a positive probability of survival. On the
 other hand, for $\ee$ small enough, if $z_m \in (z_w-\ee,z_w)$ then
 $\Lambda(z_m)<0$ and the branching process $X_m$ is subcritical and
 thus almost surely goes extinct.

 \subsection{Example}
 Let us consider the Lysis/Lysogeny Model in a Fluctuating Environment
of  \citet{BruManSmiCarGanWed2026}.

The densities of susceptible cells $S(t)$, lysogenized cells $L(t)$, and free virus particles $V(t)$ satisfy the differential equation
\begin{align}
  \frac{dS}{dt} &= \theta + r S (1 - \kappa N) - (aV + d_s) S, \\
  \frac{dL}{dt} &= r L (1 - \kappa N) + a \phi V S - (\alpha + d_L) L, \\
  \frac{dV}{dt} &= a (1 - \phi) B V S + \alpha B L - (aN + d_V) V,
\end{align}
with $N(t) = S(t) + L(t)$ the total density of cells. The random
environment is $(\theta(t))_{t \geq 0}$, a uniquely ergodic Markov
process with values in $[0,\theta_{max}]$ (which has a unique invariant probability measure). Consider as a trait the rate $z = \alpha$
of reactivation of lysogenized cells.

The random environment is thus $\omega_t=(\theta(t),S(t),N(t))_{t\ge
  0}$. Since $N(t)$ is solution of
\begin{equation}
  \frac{dN}{dt} \le rN(1-\kappa N) -  \min(d_S,\alpha +d_L)\, N \le r
  N (1-\frac{N}{K})\,,
\end{equation}
$\omega_t$ takes its value in the compact $\Srond=\ens{\tau=(\theta,S,N) : \theta \in
  [0,\theta_{max}], 0\le S\le N\le K}$. We assume that the ergodicity
of $(\theta(t))_{t\ge 0}$ ensures the ergodicity of $(\omega_t)_{t\ge 0}$.

The branching approximation of the mutant population is a Markov process $X_m(t)=(L_m(t),V_m(t))$ with a time-dependent and environment-dependent generator: if $x = (l, v)$, then
\begin{align}
  L^{X_m} f(x) &= (l r(1 - \kappa N(t)+ v a\phi S(t)) (f(x + e_L) - f(x))\\
  &+ l (\alpha + d_L) (f(x-e_L) -f(x))\\
  &+B(l \alpha + v a (1-\phi) S(t))  (f(x + e_V) - f(x)) \\
  &+(aN(t) + d_V) v (f(x-e_V) -f(x))\,,
\end{align}

with $e_L = (1, 0)$ and $e_V = (0, 1)$. The dynamics of the means is therefore given by
\begin{equation}\label{eq:meandyn}
  \frac{dY}{dt} = A(\alpha,\omega_t) Y(t),
\end{equation}
with $\tau=(\theta,S,N)$
\begin{equation}
  A(\alpha, \tau) = \begin{pmatrix}
    r(1 - \kappa N) - (\alpha + d_L) &\alpha B \\
    a \phi S & \etc{a (1 - \phi) B S - (a N + d_V)}
  \end{pmatrix}^T.
\end{equation}
The mean matrix is obtained from the preceding formula by replacing
$S$ by the mean $\esp{S(0)}>0$ and thus is seen to be irreducible. We
numerically \emph{check} by a Monte Carlo simulation that
$\esp{\Gamma^{-2}}< +\infty$, with $\Gamma=\inf_{t\in [0,1]} S(t)$. We
can apply Theorem~\ref{thm:main}.
Since 
\begin{equation}
  \frac{\partial A(\alpha, \tau)}{\partial \alpha} = \begin{pmatrix}
    -1 & B \\ 0 & 0
  \end{pmatrix}^T
\end{equation}
 we can compute the selection gradient
\begin{equation}
  \Lambda'(\alpha) = \esp{\crochet{V(\omega), \frac{\partial A(\alpha,
        \omega_0)}{\partial \alpha} U(\omega)}}=\esp{U_L(B V_V -V_L)}\,.
\end{equation}
If we assume integral separation for the dynamics \eqref{eq:meandyn},
we obtain a numerical approximation of this derivative
by solving the extended differential equation that gives the dynamics of $(Y(t), \partial_\alpha Y(t))$:
\begin{align*}
  \frac{dY}{dt} &= A(\omega_t) Y_t, \\
  \frac{d \partial_\alpha Y}{dt} &= A(\omega_t) \partial_\alpha Y(t) + \frac{\partial A(\alpha, \omega_t)}{\partial \alpha} Y(t).
\end{align*}
and computing the limit
\begin{equation}\label{eq:411}
  \Lambda'(\alpha) = \lim_{n \to +\infty} \frac{1}{t_n} \frac{\crochet{\partial_\alpha Y(t_n), \vun}}{\crochet{Y(t_n), \vun}}\,,
\end{equation}
with $t_n=n^{1/2}$.

\smallskip
The figure \ref{fig:lyse} shows the convergence of \eqref{eq:411} for
  5 different trajectories. The process $\theta(t)$ is a continuous
  time Markov chain with values in $\ens{0,\theta_m}$ and with
  invariant probability $(\pi_0,1-\pi_0)$.

  \begin{figure}
    \centering
    \begin{subfigure}{0.45\textwidth}
    \includegraphics[width=\textwidth]{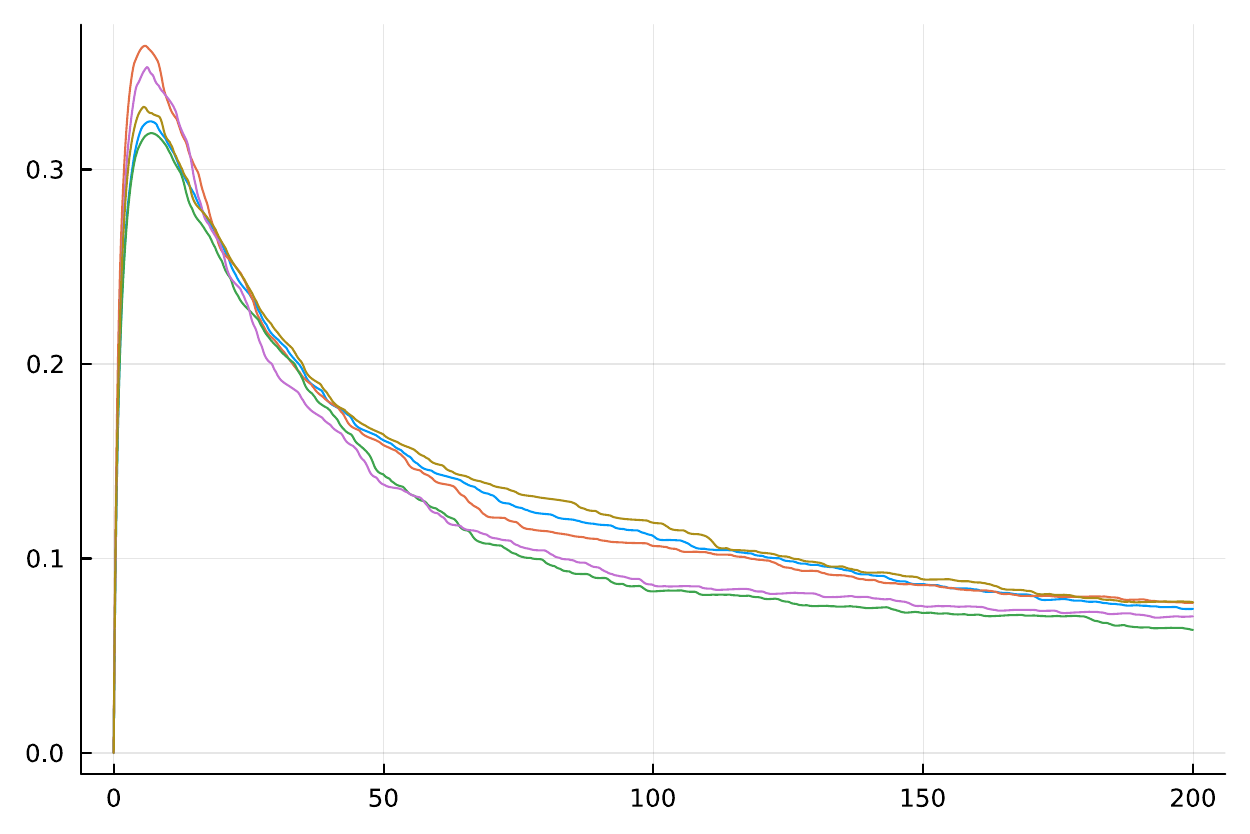} 
    \caption{}
    \label{fig:first}
\end{subfigure}
\hfill
\begin{subfigure}{0.45\textwidth}
    \includegraphics[width=\textwidth]{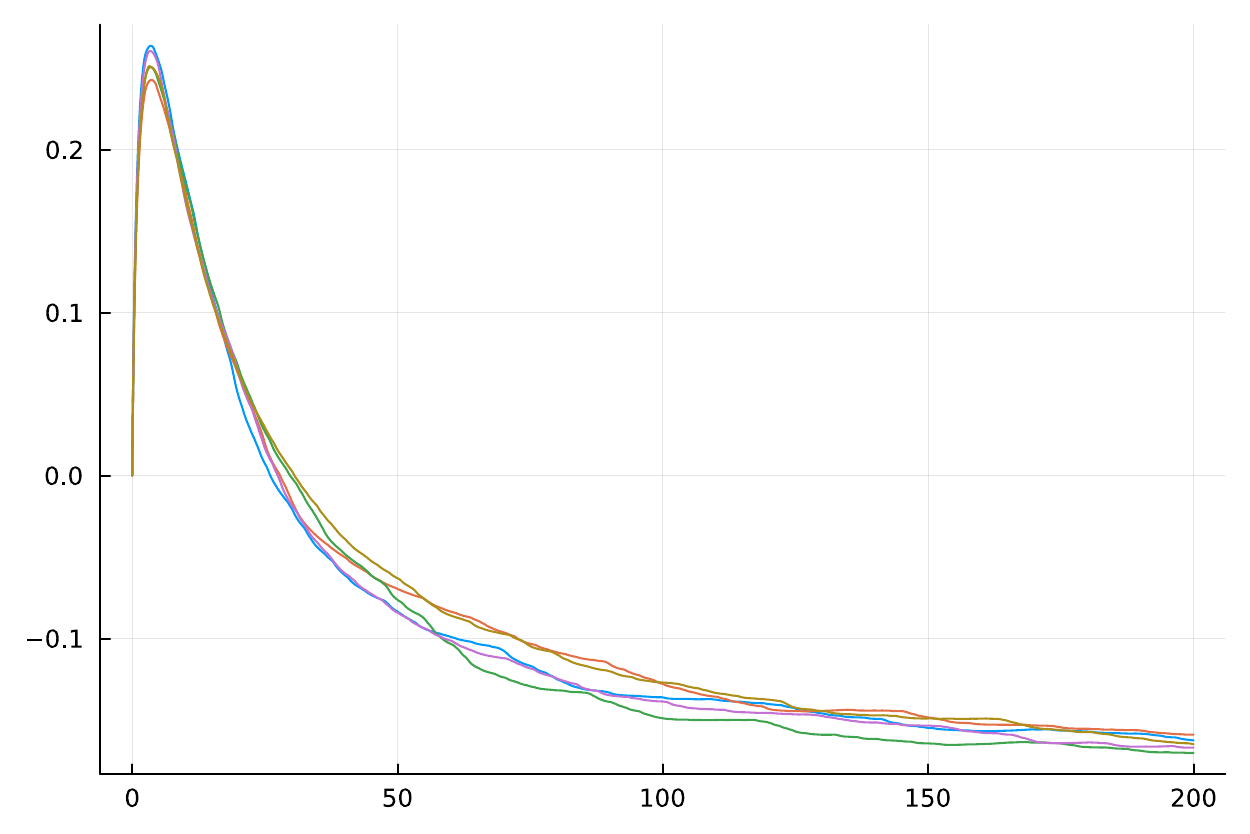}
    \caption{}
    \label{fig:second}
\end{subfigure}
    \caption{Computation of the selection gradient. We have 5
      different trajectories of the curve $t\to \frac{1}{t} \frac{\crochet{\partial_\alpha Y(t), \vun}}{\crochet{Y(t), \vun}} $. On the left graph, $\alpha=0.05$ and
      $\Lambda'(\alpha)>0$, on the right graph $\alpha=0.1$ and
      $\Lambda'(\alpha)<0$. Parameters  are taken from Table 1 of \citet{BruManSmiCarGanWed2026} : $\pi_0=0.2$, $\theta_m=250$,
      $r=1$, $\kappa=10^{-3}$, $a=10^{-4}$, $d_S=d_L=d_V=0.1$, $B=10$, $\phi=0.325$.}
    \label{fig:lyse}
  \end{figure}
\xcom{
\subsection{Two Examples of Explicit Calculation of the Leading Lyapunov Exponent in a Random Environment}
A first example is \citet{CaiGer2020}. We can see from the equation describing the model (2.1a) and its application to a structured population model in a random environment in section 4.3, that the model is diagonalizable (see equation (4.6) which describes the evolution of $I_{i, t}$, the size of the population infected by virus $x_i$) and therefore one-dimensional. In this case, we have an explicit calculation, in a very general framework, of the growth rate as a function of the trait $x$, and even in equation (4.9) of the invasion fitness $\Srond_x(y)$ of a mutant initially rare with trait $y$ in a resident population with trait $x$, from which we can immediately deduce the selection gradient $S_x = \frac{d \Srond_x(y)}{dy}_{\mid y = x}$.

A second example is \citet{BlaHerSlo2021}, where the authors explicitly compute the leading Lyapunov exponent for 2-dimensional branching processes in a random environment (see Theorems 2.12 to 2.15). They deduce a \emph{fitness advantage under fair comparison} that allows them to compare different strategies. Note that these calculations allow comparing the growth rates of different strategies, but they do not allow computing the growth rate of a strategy of type $y$ mutant initially rare, in a resident environment of type $x$, and therefore not the selection gradient.
}
\section{The integral formula for the derivative in the periodic case}
Here, $S = \etc{0, 1} = \R/\Z$ and under $\PP_s$, we have $\omega(t) = t + s \mod (1)$. This is a uniquely ergodic Feller process with invariant measure $\mu$, the uniform measure on $S$, and thus for the process $(\omega_t)_{t \geq 0}$ to be stationary ergodic, we must consider the probability measure
\begin{equation}
  \PP(A) = \int_0^1 \PP_s(A) \, ds.
\end{equation}
Assume that $A:I\times \R \to \Mrond$ is continous, ith $I=[a,b]$,
$a<b$ and 1-periodic : $A(z,1+s)=A(z,s)$ such that its restriction
$A_{\mid S}:I\times S \to \Mrond$ statisfies Hypothesis \ref{hyp:un}.

We denote by $\phi_A$ the fundamental solution associated with the matrix function $A$:
\begin{equation}
  \frac{d \phi_A(t)}{dt} = A(t) \phi_A(t), \quad \phi_A(0) = I_d.
\end{equation}
Under $\PP_s$, we have $A(\omega_t) = A(t + s)$ so that
\begin{equation}
  \frac{d}{dt} \Phi(t, \omega) = A(t + s) \Phi(t, \omega),
\end{equation}
that is, under $\PP_s$, we have
\begin{equation}
  \Phi(t, \omega) = \phi_{A(s + \cdot)}(t) = \phi_A(t + s) \phi_A(s)^{-1}.
\end{equation}
We fix $ z \in I$ and we shall omit the subscript for
readability. According to Theorem \ref{thm:main}
there exist two strictly positive random column vectors $U(\omega)$ and row vectors $V(\omega)$ such that $\crochet{\vun, U(\omega)} = \crochet{V(\omega), U(\omega)} = 1$ and
\begin{equation}
  \Phi(t, \omega) U(\omega) = q_t(\omega) U(\theta_t \omega) \quad \PP \, \text{a.s.}
\end{equation}
Consequently, for almost all $s$, under $\PP_s$,
\begin{equation}
  \phi_A(t + s) \phi_A(s)^{-1} U(\omega) = q_t(\omega) U(\theta_t \omega).
\end{equation}
By the continuity of the flow $\phi_A(t)$ and thus that of $t \to q_t(\omega)$, this identity holds for all $s \in [0, 1]$. We will specialize it to $t = 1$. Since $A(t + 1) = A(t)$, we have $\phi_A(1 + s) = \phi_A(s) \phi_A(1)$. Moreover, $\theta_1 \omega = \omega$, so for all $s$ under $\PP_s$,
\begin{equation}
  \phi_A(s) \phi_A(1) \phi_A(s)^{-1} U(\omega) = q_1(\omega) U(\omega).
\end{equation}
We consider the Perron-Frobenius eigenvectors associated with the
positive matrix $\phi_A(1)$: for $\lambda = \rho(\phi_A(1))$ (the
spectral radius), there exist \emph{unique}strictly positive vectors $u_0, v_0$ such that $\crochet{u_0,\vun}=\crochet{u_0,v_0}=1$ and
\begin{equation}
  \phi_A(1) u_0 = \lambda u_0, \quad \text{and} \quad  \phi_A(1)^* v_0 = \lambda v_0.
\end{equation}
By uniqueness, we therefore obtain that under $\PP_s$, a.s., $q_1(\omega) = \lambda$ and for a constant $c > 0$, $\phi_A(s)^{-1} U(\omega) = cu_0$. This implies that
\begin{equation}
  U(\omega) = \frac{\phi_A(s) u_0}{\valabs{\phi_A(s) u_0}}.
\end{equation}

Similarly, from the equation
\begin{equation}
  \Phi^*(t,\omega)V(\theta_t \omega)  = q_t(\omega) V(\omega) \quad \PP_\mu \, \text{a.s.},
\end{equation}
we deduce by taking $t = 1$ that for all $s$ under $\PP_s$, almost surely,
\begin{equation}
  \phi_A^*(s)^{-1} \phi_A^*(1) \phi_A^*(s) V(\omega) = q_1(\omega) V(\omega),
\end{equation}
and thus that for a constant $c > 0$, we have $V(\omega) = c v_0
\phi_A^*(s)^{-1} v_0$. Since $1 = V(\omega) U(\omega)$, it follows that
\begin{equation}
  V(\omega) = \valabs{\phi_A(s) u_0} \phi_A^*(s)^{-1} v_0 \quad \PP_s \, \text{a.s.}
\end{equation}
The function $u(s) = \phi_A(s) u_0$ is a positive $1$-periodic solution of
\begin{equation*}
  \frac{du}{ds} = A(s) u(s).
\end{equation*}
  The function $v(s) = \phi_A^*(s)^{-1} v_0$ is a positive
  $1$-periodic solution of 
  \begin{equation*}
    \frac{dv}{ds} = -A^*(s) v(s)
  \end{equation*}
  and we have $\crochet{u(s),v(s)}=1$ for all $s$. 
Consequently, the selection gradient is 
\begin{align*}
  \Lambda'(z) &= \esperance{}{\crochet{V(\omega), \partial_z A(\omega_0) U(\omega)}} \\
  &= \int_0^1 \esperance{s}{\crochet{V(\omega), \partial_z A(\omega_0) U(\omega)}} \, ds \\
  &= \int_0^1 \crochet{v(s), \partial_z A(s) u(s)} \, ds.
\end{align*}
This is indeed the formula for the selection gradient in a periodic environment of \citet{LioGan2022}.

\xcom{
\section{An Alternative Proof of the Formula for the Derivative}
This alternative proof is based on the very beautiful paper by \citet{Ruelle79}. In the simplified framework where $\Prond$ is the open set of measurable functions $T: \Omega \to \Mrond_m(\R)$ such that
\begin{equation}
  T(\omega)(\rdp) \subset \ens{0} \cup \rdpp,
\end{equation}
Theorem 3.1 of \citet{Ruelle79} establishes that the mapping $T \to \Xrond(T)$, which assigns to $T$ its leading Lyapunov exponent, is real analytic on $\Prond$. Moreover, if the invariant subspaces of $T$ and $T^*$ are generated by the random variables $U(\omega), V(\omega)$ of $\rdpp$, we have the following formula for the differential (formula (3.6) of \citet{Ruelle79}):
\begin{equation}\label{eq:diffruelle}
  D\Xrond(T)(S) = \esp{\frac{\crochet{V(\theta \omega), S(\omega) U(\omega)}}{\crochet{V(\theta \omega), T(\omega) U(\omega)}}}.
\end{equation}
We will apply this result to $T(\omega) = \Phi(z, 1, \omega)$. By the cocycle property, keeping Ruelle's notations as much as possible,
\begin{align*}
  T^n_\omega &= T(\theta^{n-1} \omega) \cdots T(\theta \omega) T(\omega) \\
  &= \Phi(z, 1, \theta^{n-1} \omega) \cdots \Phi(z, 1, \theta \omega) \Phi(z, 1, \omega) \\
  &= \Phi(z, n, \theta^n \omega).
\end{align*}
Consequently, the leading Lyapunov exponent is the almost sure limit
\begin{equation*}
  \Xrond(T) = \lim_{n \to +\infty} \frac{1}{n} \log \norme{T^n_\omega} = \lim_{n \to +\infty} \frac{1}{n} \norme{\Phi(z, 1, \theta^n \omega)} = \Lambda(z).
\end{equation*}
It remains for us to identify the term $S$ corresponding to the consequence on $\phi(z, 1, \omega)$ of an infinitesimal variation of $A(z, \omega_t)$, and then to compute the expectation giving the derivative.

First, note that if $\phi_M$ is the fundamental solution associated with the continuous matrix function $t \to M(t)$:
\begin{equation*}
  \frac{d \phi_M(t)}{dt} = M(t) \phi_M(t), \quad \phi_M(0) = I_d,
\end{equation*}
then we have the integration by parts formula
\begin{equation*}
  \phi_{M+N}(t) = \phi_M(t) + \int_0^t \phi_M(t-s) N(s) \phi_{M+N}(s) \, ds.
\end{equation*}

We deduce that the derivative of the function $z \to \Phi(z, t, \omega)$ is
\begin{equation*}
  \partial_z \Phi(z, t, \omega) = \int_0^t \Phi(z, t-s, \theta_s \omega) \partial_z A(z, \omega_s) \Phi(z, s, \omega) \, ds.
\end{equation*}
Consequently, the function $z \to \Lambda(z)$ is differentiable with derivative
\begin{equation*}
  \Lambda'(z) = \esp{\frac{\crochet{V_z(\theta \omega), \partial_z \Phi(1, z, \omega) U_z(\omega)}}{\crochet{V_z(\theta \omega), \Phi(z, 1, \omega) U_z(\omega)}}}.
\end{equation*}

To conclude our proof, it remains to show that this expression coincides with $\esp{\crochet{V_z(\omega), A(z, \omega_0) U_z(\omega)}}$. Since this is a calculation at fixed $z$, we will sometimes omit the variable $z$ to simplify the notations at the end of this demonstration.

Observe that
\begin{equation*}
  \crochet{V(\theta \omega), \Phi(1, \omega) U(\omega)} = \crochet{V(\theta \omega), q(1, \omega) U(\theta \omega)} = q(1, \omega) \crochet{V(\theta \omega), U(\theta \omega)} = q(1, \omega).
\end{equation*}
Then,
\begin{align*}
  \crochet{V_z(\theta \omega), \partial_z \Phi(1, z, \omega) U_z(\omega)} &= \int_0^1 \crochet{V_z(\theta \omega), \Phi(z, 1-s, \theta_s \omega) \partial_z A(z, \omega_s) \Phi(z, s, \omega) U_z(\omega)} \, ds \\
  &= \int_0^1 q(s, \omega) \crochet{V_z(\theta \omega), \Phi(z, 1-s, \theta_s \omega) \partial_z A(z, \omega_s) U_z(\theta_s \omega)} \, ds \\
  &= \int_0^1 q(s, \omega) \crochet{\Phi^*(z, 1-s, \theta_s \omega) V_z(\theta \omega), \partial_z A(z, \omega_s) U_z(\theta_s \omega)} \, ds \\
  &= \int_0^1 q(s, \omega) q(1-s, \theta_s \omega) \crochet{V_z(\theta_s \omega), \partial_z A(z, \omega_s) U_z(\theta_s \omega)} \, ds \\
  &= q(1, \omega) \int_0^1 \crochet{V_z(\theta_s \omega), \partial_z
    A(z, \omega_s) U_z(\theta_s \omega)} \, ds, 
\end{align*}
since by the cocycle property, $q(1, \omega) = q(s, \omega) q(1-s, \theta_s \omega)$. Consequently, by stationarity,
\begin{align*}
  \Lambda'(z) &= \esp{\int_0^1 \crochet{V_z(\theta_s \omega), \partial_z A(z, \omega_s) U_z(\theta_s \omega)} \, ds} \\
  &= \int_0^1 \esp{\crochet{V_z(\theta_s \omega), \partial_z A(z, \omega_s) U_z(\theta_s \omega)}} \, ds \\
  &= \esp{\crochet{V_z(\omega), A(z, \omega_0) U_z(\omega)}}.
\end{align*}
}
\bigskip
\thanks{{\it Acknowledgments.} The author thanks the Centre Henri Lebesgue ANR-11-LABX-0020-01 for its support during the conception of this work.}

\bibliographystyle{plainnat}
\bibliography{mathbio,nbio}

\end{document}